\documentclass[a4paper,10pt]{amsart}
\usepackage[english]{babel}
\usepackage[english]{translator}
\usepackage[T1]{fontenc}
\usepackage[utf8]{inputenc}
\usepackage{amsmath}
\usepackage{amssymb}
\usepackage{amsfonts}
\usepackage{amstext}
\usepackage{amsthm}
\usepackage{bbm}

\usepackage{graphicx}
\usepackage{color}
\usepackage{psfrag}
\usepackage{subfigure}
\usepackage{float}
\usepackage{marvosym}

\usepackage{hyperref}
\usepackage[
  hmarginratio={1:1},     % equal left and right margins
  vmarginratio={1:1},     % equal top and bottom margins
  textwidth=420pt,        % new text width
  heightrounded,          % always useful
]{geometry}

\hypersetup{
  pdffitwindow=false,
  pdfhighlight=/O,
  pdfnewwindow,
  colorlinks=true,
  citecolor=red,            % Farbe der Zitate(Buchverweise)
  linkcolor=blue,             % Farbe der Hyperref-Links
  menucolor=blue,            % color of Acrobat Reader menu buttons
  urlcolor=blue,             % Farbe der Internetadressverweise
  pdfpagemode=UseOutlines,
  bookmarksnumbered=true,
  linktocpage,
  pdftitle={A new \texorpdfstring{$L^p$}{Lp}-Antieigenvalue Condition for Ornstein-Uhlenbeck Operators in \texorpdfstring{$L^p(\mathbb{R}^d,\mathbb{C}^N)$}{Lp(Rd,CN)}},
  pdfsubject={complex-valued Ornstein-Uhlenbeck operator, \texorpdfstring{$L^p$}{Lp}-dissipativity, \texorpdfstring{$L^p$}{Lp}-antieigenvalue condition, applications to rotating waves},
  pdfauthor={Denny Otten},
  pdfkeywords={complex-valued Ornstein-Uhlenbeck operator, \texorpdfstring{$L^p$}{Lp}-dissipativity, \texorpdfstring{$L^p$}{Lp}-antieigenvalue condition, applications to rotating waves},
  pdfcreator={},
  pdfproducer={LaTeX mit hyperref}
}

\newcommand{\D}{\mathcal{D}}              % Definitionsbereich
\newcommand{\K}{\mathbb{K}}               % Koerper
\newcommand{\C}{\mathbb{C}}               % komplexe Zahlen
\newcommand{\R}{\mathbb{R}}               % reelle Zahlen
\newcommand{\N}{\mathbb{N}}                % natuerliche Zahlen
\renewcommand{\S}{\mathcal{S}}              % Schwartz Raum
\renewcommand{\Re}{\mathrm{Re}\,}          % Realteil
\renewcommand{\Im}{\mathrm{Im}\,}          % Imaginaerteil
\renewcommand{\L}{\mathcal{L}}             % L Operator
\newcommand{\A}{\mathcal{A}}             % L Operator
\newcommand{\begriff}[1]{\textbf{#1}}

\makeatletter
  \def\section{\@startsection{section}{1}%
    \z@{.7\linespacing\@plus\linespacing}{.5\linespacing}%
    {\normalfont\LARGE\bfseries}}
\makeatother
\newcommand{\sect}
{
  \setcounter{equation}{0}
  \setcounter{figure}{0}
  \section
}
\newcommand{\enum}[1]{\textnormal{(\textbf{#1})}}

\theoremstyle{plain}
\newtheorem{definition}{Definition}[section]
\newtheorem{theorem}[definition]{Theorem}

\newtheorem{assumption}[definition]{Assumption}

\theoremstyle{definition}

\allowdisplaybreaks

\begin{document}
%  -----------
% |   Title   |
%  -----------
\title[A new $L^p$-Antieigenvalue Condition for\\Ornstein-Uhlenbeck Operators]{A new $L^p$-Antieigenvalue Condition for\\Ornstein-Uhlenbeck Operators}
%\maketitle
%\setlength{\parindent}{0pt}
\vspace*{0.75cm}
\begin{center}
%\normalfont\LARGE\bfseries{\shorttitle}\\
\normalfont\huge\bfseries{\shorttitle}\\
\vspace*{0.25cm}
%\Large\bfseries\MakeUppercase{\shorttitle}
\end{center}

%  -------------
% |   Authors   |
%  -------------
\vspace*{0.5cm}
\noindent
\hspace*{4.2cm}
\textbf{Denny Otten}$\footnote[1]{e-mail: \textcolor{blue}{dotten@math.uni-bielefeld.de}, phone: \textcolor{blue}{+49 (0)521 106 4784}, \\
                                 fax: \textcolor{blue}{+49 (0)521 106 6498}, homepage: \url{http://www.math.uni-bielefeld.de/~dotten/}, \\
                                 supported by CRC 701 'Spectral Structures and Topological Methods in Mathematics'.}$ \\
\hspace*{4.2cm}
Department of Mathematics \\
\hspace*{4.2cm}
Bielefeld University \\
\hspace*{4.2cm}
33501 Bielefeld \\
\hspace*{4.2cm}
Germany

%  ----------
% |   Date   |
%  ----------
\vspace*{0.8cm}
\noindent
\hspace*{4.2cm}
Date: \today
\normalparindent=12pt

%  --------------
% |   Abstract   |
%  --------------
\vspace{0.4cm}
%\begin{abstract}
\noindent
\begin{center}
\begin{minipage}{0.9\textwidth}
  {\small
  \textbf{Abstract.} 
  In this paper we study perturbed Ornstein-Uhlenbeck operators
  \begin{align*}
    \left[\mathcal{L}_{\infty} v\right](x)=A\triangle v(x) + \left\langle Sx,\nabla v(x)\right\rangle-B v(x),\,x\in\mathbb{R}^d,\,d\geqslant 2,
  \end{align*}
  for simultaneously diagonalizable matrices $A,B\in\mathbb{C}^{N,N}$. The unbounded drift term is defined by a skew-symmetric matrix
  $S\in\mathbb{R}^{d,d}$. Differential operators of this form appear when investigating rotating waves in time-dependent reaction diffusion systems. 
  As shown in a companion paper, one key assumption to prove resolvent estimates of $\mathcal{L}_{\infty}$ in $L^p(\R^d,\C^N)$, $1<p<\infty$, is 
  the following $L^p$-dissipativity condition
  \begin{align*}
    |z|^2\Re\left\langle w,Aw\right\rangle + (p-2)\Re\left\langle w,z\right\rangle\Re\left\langle z,Aw\right\rangle\geqslant\gamma_A |z|^2|w|^2\;\forall\,z,w\in\C^N
  \end{align*}
  for some $\gamma_A>0$. We prove that the $L^p$-dissipativity condition is equivalent to a new $L^p$-antieigenvalue condition
  \begin{align*}
    A\text{ invertible}\quad\text{and}\quad\mu_1(A)>\frac{|p-2|}{p},\,1<p<\infty,\,\mu_1(A)\text{ first antieigenvalue of $A$,}
  \end{align*}
  which is a lower $p$-dependent bound of the first antieigenvalue of the diffusion matrix $A$. This relation provides a complete algebraic characterization 
  and a geometric meaning of $L^p$-dissipativity for complex-valued Ornstein-Uhlenbeck operators in terms of the antieigenvalues of $A$. 
  The proof is based on the method of Lagrange multipliers. We also discuss several special cases in which the first antieigenvalue can be given explicitly.
  }
\end{minipage}
\end{center}
%\end{abstract}



%\tableofcontents

%---------------------------------------------------------------------------------------------------------------------------------------------------
%
%  SECTION 1: (Introduction)
%
%---------------------------------------------------------------------------------------------------------------------------------------------------
\sect{Introduction}
\label{sec:Introduction}
%---------------------------------------------------------------------------------------------------------------------------------------------------

% Introduction (-> Einleitung)
In this paper we study $L^p$-dissipativity of differential operators of the form
\begin{align*}
  \left[\L_{\infty}v\right](x) := A\triangle v(x) + \left\langle Sx,\nabla v(x)\right\rangle - Bv(x),\,x\in\R^d,\,d\geqslant 2,
\end{align*}
for simultaneously diagonalizable matrices $A,B\in\C^{N,N}$ with $\Re\sigma(A)>0$ and a skew-symmetric matrix $S\in\R^{d,d}$.

Introducing the complex Ornstein-Uhlenbeck operator, \cite{UhlenbeckOrnstein1930},
\begin{align*}
  \left[\L_0 v\right](x) := A\triangle v(x) + \left\langle Sx,\nabla v(x)\right\rangle,\,x\in\R^d,
\end{align*}
with \begriff{diffusion term} and \begriff{drift term} given by
\begin{align*}
  A\triangle v(x):=A\sum_{i=1}^{d}\frac{\partial^2}{\partial x_i^2}v(x)\quad\text{and}\quad\left\langle Sx,\nabla v(x)\right\rangle:=\sum_{i=1}^{d}(Sx)_i\frac{\partial}{\partial x_i}v(x),
\end{align*}
we observe that the operator $\L_{\infty}=\L_0-B$ is a constant coefficient perturbation of $\L_0$. Our interest is in skew-symmetric matrices $S=-S^T$, 
in which case the drift term is rotational containing angular derivatives
\begin{align*}
  \left\langle Sx,\nabla v(x)\right\rangle=\sum_{i=1}^{d-1}\sum_{j=i+1}^{d}S_{ij}\left(x_j\frac{\partial}{\partial x_i}-x_i\frac{\partial}{\partial x_j}\right)v(x).
\end{align*}
Such differential operators arise when investigating exponential decay of rotating waves in reaction diffusion systems, see \cite{BeynLorenz2008,Otten2014}. 
The operator $\L_{\infty}$ appears as a far-field linearization at the solution of the nonlinear problem $\L_0 v=f(v)$. The results of this paper are 
crucial for dealing with the nonlinear case, see \cite{Otten2014}. 

An essential ingredient to treat the nonlinear case is to prove $L^p$-resolvent estimates for $\L_{\infty}$, \cite[Theorem 4.4]{Otten2015a}. Such estimates are used to 
solve the identification problem for $\L_{\infty}$, \cite[Theorem 5.1]{Otten2015a}. One key assumption to prove resolvent estimates of $\mathcal{L}_{\infty}$ in $L^p(\R^d,\C^N)$, 
$1<p<\infty$, is the following $L^p$-dissipativity condition
\begin{align*}
  |z|^2\Re\left\langle w,Aw\right\rangle + (p-2)\Re\left\langle w,z\right\rangle\Re\left\langle z,Aw\right\rangle\geqslant\gamma_A |z|^2|w|^2\;\forall\,z,w\in\C^N
\end{align*}
for some $\gamma_A>0$. In general, it is not easy to characterize the class of matrices $A$ which satisfy this algebraic condition. Only in a few special cases, 
e.g. in the scalar case $N=1$ or for general $N$ and $p=2$, one can check the validity directly.

% Aim of the paper (roughly descibed) (-> Ziel grob formuliert)
The aim of this paper is to prove that the $L^p$-dissipativity condition is equivalent to a new $L^p$-antieigenvalue condition, namely
\begin{align*}
  A\text{ invertible}\quad\text{and}\quad\mu_1(A)>\frac{|p-2|}{p},\,1<p<\infty,\,\mu_1(A)\text{ first antieigenvalue of $A$,}
\end{align*}
which is a lower $p$-dependent bound of the first antieigenvalue of the diffusion matrix $A$. This criterion implies an upper $p$-dependent bound for the maximal 
(real) angle of $A$
\begin{align*}
  \Phi_{\R}(A):=\cos^{-1}\left(\mu_1(A)\right)<\cos^{-1}\left(\frac{|p-2|}{p}\right)\in]0,\frac{\pi}{2}],\quad 1<p<\infty.
\end{align*}
The relation between the $L^p$-dissipativity and the $L^p$-antieigenvalue condition seems to be new in the literature and is proved in Theorem 
\ref{thm:LpDissipativityCondition}. It provides a complete algebraic characterization and a nice geometric meaning of $L^p$-dissipativity for 
complex-valued Ornstein-Uhlenbeck operators in terms of the antieigenvalues of $A$. The proof is based on the method of Lagrange multipliers 
and requires to destinguish between the cases $A\in\R^{N,N}$ and $A\in\C^{N,N}$. We also discuss several special cases in which the first antieigenvalue 
can be given explicitly.

% Procedure of the paper (-> Vorgehensweise) 

% Applications (-> Anwendungen)

% Literature (-> Literatur)
$L^p$-dissipativity of second order differential operators in the scalar but more general case has been analyzed by Cialdea and Maz'ya in \cite{Cialdea2010, CialdeaMazya2009, 
Cialdea2009, CialdeaMazya2005}. General theory of antieigenvalues has been developed by Gustafson in \cite{Gustafson1968,Gustafson1972} and independently by Kre{\u\i}n in 
\cite{Krein1969}. Explicit representations of antieigenvalues have been established for Hermitian positive definite operators by Mirman in \cite{Mirman1983} and by Horn and Johnson in 
\cite{HornJohnson2013}, and for (strictly) accretive normal operators by Seddighin and Gustafson in \cite{GustafsonSeddighin2005,Seddighin2005,GustafsonSeddighin1993,GustafsonSeddighin1989},
by Davis in \cite{Davis1980} and by Mirman in \cite{Mirman1983}. Approximation results and the computation of antieigenvalues have been analyzed by Seddighin in 
\cite{Seddighin2012} and \cite{Seddighin2005a, Seddighin2003}, respectively. For general theory of antieigenvalues and its application to operator theory, numerical analysis, 
wavelets, statistics, quantum mechanics, finance and optimization we refer to the book by Gustafson, \cite{Gustafson2012}. Further applications are treated in 
\cite{Gustafson1968,Gustafson1968b,GustafsonRao1997,Gustafson1997}. There are some extensions of the antieigenvalue theory to higher antieigenvalues, see \cite{GustafsonSeddighin1993, 
GustafsonSeddighin2005}, to joint antieigenvalues, see \cite{Seddighin2005}, to symmetric antieigenvalues, see \cite{PaulDas2010}, and to $\theta$-antieigenvalues, see 
\cite{PaulDasDebnath2013}. Historical background material can be found in \cite{Gustafson2012, Gustafson1994}. The method of Lagrange multipliers, that is necessary 
to prove our main result, is also used in \cite{GustafsonSeddighin1989}. The $L^p$-dissipativity condition can be found in \cite{Otten2014,Otten2015a} and is used 
to prove resolvent estimates for complex Ornstein-Uhlenbeck systems.

The results from Section \ref{sec:TheLpAntieigenvalueCondition} and \ref{sec:SomeExplicitRepresentationsOfTheFirstAntieigenvalue} are directly based on the 
PhD thesis \cite{Otten2014}.

%---------------------------------------------------------------------------------------------------------------------------------------------------
%
%  SECTION 2: (Assumptions and outline of results)
%
%---------------------------------------------------------------------------------------------------------------------------------------------------
\sect{Assumptions and outline of results}
\label{sec:Preliminaries}
%---------------------------------------------------------------------------------------------------------------------------------------------------

Consider the differential operator 
\begin{align*}
  \left[\L_{\infty}v\right](x) := A\triangle v(x) + \left\langle Sx,\nabla v(x)\right\rangle - Bv(x),\,x\in\R^d,\,d\geqslant 2,
\end{align*}
for some matrices $A,B\in\C^{N,N}$ and $S\in\R^{d,d}$.

The following conditions will be needed in this paper and relations among them will be discussed below.
%  ----------------
% | Assumption 2.1 |
%  ----------------
\begin{assumption} 
  \label{ass:Assumption1}
  Let $A,B\in\K^{N,N}$ with $\K\in\{\R,\C\}$ and $S\in\R^{d,d}$ be such that
  \begin{flalign}
    &\text{$A$ and $B$ are simultaneously diagonalizable (over $\C$)},         \tag{A1}\label{cond:A8B} &\\
    &\Re\sigma(A)>0, \tag{A2}\label{cond:A2} &\\
    &\text{There exists some $\beta_A>0$ such that} \tag{A3}\label{cond:A3} &\\
    &\quad\Re\left\langle w,Aw\right\rangle\geqslant\beta_A\;\forall\,w\in\K^N,\,|w|=1, \nonumber &\\
    &\text{There exists some $\gamma_A>0$ such that} \tag{A4}\label{cond:A4DC} &\\
    &\quad|z|^2\Re\left\langle w,Aw\right\rangle + (p-2)\Re\left\langle w,z\right\rangle\Re\left\langle z,Aw\right\rangle\geqslant\gamma_A |z|^2|w|^2\;\forall\,z,w\in\K^N \nonumber &\\
    &\text{for some $1<p<\infty$,} \nonumber &\\
    &\text{Case ($N=1$, $\K=\R$):} \tag{A5}\label{cond:A4} \\
    &\quad A=a>0, \nonumber &\\
    &\text{Cases ($N\geqslant 2$, $\K=\R$) and ($N\geqslant 1$, $\K=\C$):} \nonumber &\\
    &\quad A\text{ invertible}\quad\text{and}\quad\mu_1(A)>\frac{|p-2|}{p}\text{ for some $1<p<\infty$,} \nonumber &\\
    &\text{$S$ is skew-symmetric}.        \tag{A6}\label{cond:A5} &
  \end{flalign}
\end{assumption}

Assumption \eqref{cond:A8B} is a \textbf{system condition} and ensures that some results for scalar equations can be extended to system cases. This condition was used 
in \cite{Otten2014,Otten2014a} to derive an explicit formula for the heat kernel of $\L_{\infty}$. It is motivated by the fact that a transformation of a scalar 
complex-valued equation into a $2$-dimensional real-valued system always implies two (real) matrices $A$ and $B$ that are simultaneously diagonalizable (over $\C$).
The \textbf{positivity condition} \eqref{cond:A2} guarantees that the diffusion part $A\triangle$ is an elliptic operator. It requires that all eigenvalues $\lambda$ of $A$ 
are contained in the open right half-plane $\C_+:=\{\lambda\in\C\mid\Re\lambda>0\}$, where $\sigma(A)$ denotes the spectrum of $A$. Condition \eqref{cond:A2} guarantees 
that $A^{-1}$ exists and states that $-A$ is a stable matrix. To discuss the \textbf{strict accretivity condition} \eqref{cond:A3}, we recall the following definition, 
from \cite{Gustafson1968,Gustafson2012}.

%  ----------------
% | Definition 2.2 |
%  ----------------
\begin{definition}\label{def:AccretiveDissipativeMatrix}
  Let $A\in\K^{N,N}$ with $\K\in\{\R,\C\}$ and $N\in\N$, then $A$ is called
  \begriff{accretive} (or \begriff{strictly accretive}), if
  \begin{align*}
    \underset{\substack{w\in\K^N\\|w|=1}}{\inf}\Re\left\langle w,Aw\right\rangle\geqslant 0
    \quad(\text{or}\quad\underset{\substack{w\in\K^N\\|w|=1}}{\inf}\Re\left\langle w,Aw\right\rangle > 0),
  \end{align*}
  and \begriff{dissipative} (or \begriff{strictly dissipative}), if
  \begin{align*}
    \underset{\substack{w\in\K^N\\|w|=1}}{\sup}\Re\left\langle w,Aw\right\rangle\leqslant 0 
    \quad(\text{or}\quad\underset{\substack{w\in\K^N\\|w|=1}}{\sup}\Re\left\langle w,Aw\right\rangle < 0 ).
  \end{align*}
  %\begin{itemize}
  %  \item \begriff{accretive}, if $\underset{\substack{w\in\K^N\\|w|=1}}{\inf}\Re\left\langle w,Aw\right\rangle\geqslant 0$,
  %  \item \begriff{strictly accretive}, if $\underset{\substack{w\in\K^N\\|w|=1}}{\inf}\Re\left\langle w,Aw\right\rangle > 0$,
  %  \item \begriff{dissipative}, if $\underset{\substack{w\in\K^N\\|w|=1}}{\sup}\Re\left\langle w,Aw\right\rangle\leqslant 0$,
  %  \item \begriff{strictly dissipative}, if $\underset{\substack{w\in\K^N\\|w|=1}}{\sup}\Re\left\langle w,Aw\right\rangle < 0$.
  %\end{itemize}
  For Hermitian matrices $A$, replace accretive (strictly accretive, dissipative and strictly dissipative) by \begriff{positive semi-definite} (\begriff{positive definite}, 
  \begriff{negative semi-definite} and \begriff{negative definite}).
\end{definition}

Condition \eqref{cond:A3} states that the matrix $A$ is strictly accretive, which is more restrictive than \eqref{cond:A2}. In \eqref{cond:A3}, 
$\left\langle u,v\right\rangle:=\overline{u}^T v$ denotes the standard inner product on $\K^N$. Note that condition \eqref{cond:A2} is satisfied 
if and only if
\begin{align*}
  \exists\,\left[\cdot,\cdot\right]\text{ inner product on $\K^N$}:\quad\Re\left[w,Aw\right]\geqslant\beta_A>0\;\forall\,w\in\K^N,\,\left[w,w\right]=1,
\end{align*}
but it does not imply $\left[\cdot,\cdot\right]=\left\langle\cdot,\cdot\right\rangle$. Condition \eqref{cond:A3} ensures that the differential operator $\L_{\infty}$ 
is closed on its (local) domain $\D^p_{\mathrm{loc}}(\L_0)$. The \textbf{$L^p$-dissipativity condition} \eqref{cond:A4DC} seems to be new in the literature and is used 
to prove $L^p$-resolvent estimates for $\L_{\infty}$ in \cite{Otten2014,Otten2015a}. Condition \eqref{cond:A4DC} is more restrictive than \eqref{cond:A3} and imposes 
additional requirements on the spectrum of $A$. $L^p$-dissipativity results for scalar differential operators of the form
\begin{align*}
  \L v = \nabla^T\left(Q\nabla v\right) + b^T\nabla v + av,\,x\in\Omega
\end{align*}
has been established in \cite{CialdeaMazya2005} for constant coefficients $Q\in\C^{d,d}$, $b\in\C^d$, $a\in\C$ with $\Omega\subseteq\R^d$ open, 
and for variable coefficients $Q_{ij},b_j\in C^1(\overline{\Omega},\C)$, $a\in C^0(\overline{\Omega},\C)$, $i,j=1,\ldots d$, with $\Omega\subset\R^d$ bounded. 
In the scalar complex case with $A=\alpha\in\C$ and $B=\delta\in\C$, the choice
\begin{align*}
  Q=\alpha I_d,\;b=Sx,\;a=\delta,\;\Omega=\R^d
\end{align*}
implies $\L_{\infty}=\L$ and leads to a differential operator with variable coefficients but on an unbounded domain. Thus, the $L^p$-dissipativity 
of $\L_{\infty}$ has not been treated in \cite{CialdeaMazya2005}, neither for the system nor for the scalar case. Therefore, the $L^p$-dissipativity 
condition \eqref{cond:A4DC}, which has been established in \cite{Otten2014,Otten2015a}, can not be deduced from \cite{CialdeaMazya2005}. 
Recall the following definition from \cite{Gustafson1968,Gustafson2012}.

%  ----------------
% | Definition 2.3 |
%  ----------------
\begin{definition}\label{def:AntieigenvalueAntieigenvectorAngle}
  Let $A\in\K^{N,N}$ with $\K\in\{\R,\C\}$ and $N\in\N$. Then we define by
  \begin{align}
    \label{equ:FirstAntieigenvalue}
    \mu_1(A):=\inf_{\substack{w\in\K^N\\w\neq 0\\Aw\neq 0}}\frac{\Re\left\langle w,Aw\right\rangle}{|w||Aw|}
             =\inf_{\substack{w\in\K^N\\|w|=1\\Aw\neq 0}}\frac{\Re\left\langle w,Aw\right\rangle}{|Aw|}
  \end{align}
  the \begriff{first antieigenvalue of $A$}. A vector $0\neq w\in\K^N$ with $Aw\neq 0$ for which the infimum is attained, 
  is called an \begriff{antieigenvector of $A$}. Moreover, we define the \begriff{(real) angle of $A$} by
  \begin{align*}
    \Phi_{\R}(A):=\cos^{-1}\left(\mu_1(A)\right).
  \end{align*}
\end{definition}

The expression for $\mu_1(A)$ is also sometimes denoted by $\cos A$ (and by $\cos(\Phi_{\R}(A))$) and is called the \begriff{cosine of $A$}. 
It was introduced simultaneously by Gustafson in \cite{Gustafson1968} and by Kre{\u\i}n in \cite{Krein1969}, where the expression for 
$\mu_1(A)$ is denoted by $\mathrm{dev}\, A$ and is called the \begriff{deviation of $A$}. Note that the definition of the first antieigenvalue 
is not consistent in the literature, in the sense that sometimes the matrix $A$ is additionally assumed to be accretive or strictly accretive. 
Let us briefly motivate the geometric idea behind eigenvalues and antieigenvalues: Eigenvectors are those vectors that are stretched (or dilated) 
by a matrix (without any rotation). Their corresponding eigenvalues are the factors by which they are stretched. The eigenvalues may be 
ordered as a spectrum from smallest to largest eigenvalue. Antieigenvectors are those vectors that are rotated (or turned) by a matrix (without 
any stretching). Their corresponding antieigenvalues are the cosines of their associated turning angle. The antieigenvalues may be orderd 
from largest to smallest turning angle. Therefore, the first antieigenvalue $\mu_1(A)$ can be considered as the cosine of the maximal turning angle 
of the matrix $A$.
%Eigenvalues are measures of the amount of dilation 
%of a matrix $A$ induces in those directions which are not turned, namely, the eigenvectors. Antieigenvalues are 
%measures of the amount of rotation of a matrix $A$ induces in those direction which are turned the most, namely, 
%the antieigenvectors.
The \textbf{$L^p$-antieigenvalue condition} \eqref{cond:A4} postulates that $\mu_1(A)$ is bounded from below by a non-negative $p$-dependent constant. 
This is equivalent to the following $p$-dependent upper bound for the (real) angle of $A$,
\begin{align*}
  \Phi_{\R}(A):=\cos^{-1}\left(\mu_1(A)\right)<\cos^{-1}\left(\frac{|p-2|}{p}\right)\in]0,\frac{\pi}{2}],\quad 1<p<\infty.
\end{align*}
In the scalar complex case $A=\alpha\in\C$, assumption \eqref{cond:A4} leads to a cone condition which requires $\alpha$ to lie in a $p$-dependent sector 
in the right half-plane, see Section \ref{sec:4.2}. The cone condition coincides with the $L^p$-dissipativity condition from \cite[Theorem 2]{CialdeaMazya2005} 
for differential operators with constant coefficients on unbounded domains and with the $L^p$-quasi-dissipativity condition from \cite[Theorem 4]{CialdeaMazya2005} 
for differential operators with variable coefficients on bounded domains. Our main result in Theorem \ref{thm:LpDissipativityCondition} shows that assumptions 
\eqref{cond:A4DC} and \eqref{cond:A4} are equivalent. Therefore, \eqref{cond:A4} can be considered as a more intuitive description of assumption \eqref{cond:A4DC}. 
For some classes of matrices, the constant $\mu_1(A)$ can be given explicitly in terms of the eigenvalues of $A$, which facilitates to check condition \eqref{cond:A4DC}. 
We summarize the following relation of assumptions \eqref{cond:A2}--\eqref{cond:A4}:
\begin{align*}
  A\text{ invertible}\Longleftarrow\text{\eqref{cond:A2}}\Longleftarrow\text{\eqref{cond:A3}}\Longleftarrow\text{\eqref{cond:A4DC}}\Longleftrightarrow\text{\eqref{cond:A4}}.
\end{align*}
The \textbf{rotational condition} \eqref{cond:A5} implies that the drift term contains only angular derivatives, which is crucial for use our results from \cite{Otten2014a}.

Moreover, let $\beta_B\in\R$ be such that
\begin{align}
  \label{equ:betaB}
  \Re\left\langle w,Bw\right\rangle\geqslant -\beta_B\;\forall\,w\in\K^N,\,|w|=1.
\end{align}
If $\beta_B\leqslant 0$, \eqref{equ:betaB} can be considered as a \textbf{dissipativity condition} for $-B$, compare Definition \ref{def:AccretiveDissipativeMatrix}.

% Definition of function spaces
We introduce Lebesgue and Sobolev spaces via
\begin{align*}
  L^p(\R^d,\K^N) :=& \left\{v\in L^1_{\mathrm{loc}}(\R^d,\K^N)\mid \left\|v\right\|_{L^p}<\infty\right\}, \\
  W^{k,p}(\R^d,\K^N) :=& \left\{v\in L^p(\R^d,\K^N)\mid D^{\beta}v\in L^p(\R^d,\K^N)\;\forall\,|\beta|\leqslant k\right\},
\end{align*}
with norms
\begin{align*}
  \left\|v\right\|_{L^p(\R^d,\K^N)} := \bigg(\int_{\R^d}\left|v(x)\right|^p dx\bigg)^{\frac{1}{p}},\quad
  \left\|v\right\|_{W^{k,p}(\R^d,\K^N)} := \bigg(\sum_{0\leqslant|\beta|\leqslant k}\left\|D^{\beta}v\right\|_{L^p(\R^d,\K^N)}^p\bigg)^{\frac{1}{p}},
\end{align*}
for every $1\leqslant p<\infty$, $k\in\N_0$ and multiindex $\beta\in\N_0^d$.

Before we give a detailed outline we briefly review and collect some results from \cite{Otten2014,Otten2014a,Otten2015a} to motivate the origin of 
the $L^p$-dissipativity condition for $\L_{\infty}$.

Assuming \eqref{cond:A8B}, \eqref{cond:A2} and \eqref{cond:A5} for $\K=\C$ it is shown in \cite[Theorem 4.2-4.4]{Otten2014}, \cite[Theorem 3.1]{Otten2014a} 
that the function $H:\R^d\times\R^d\times]0,\infty[\rightarrow\C^{N,N}$ defined by
\begin{align}
  \label{equ:HeatKernel}
  H(x,\xi,t)=(4\pi t A)^{-\frac{d}{2}}\exp\left(-Bt-(4tA)^{-1}\left|e^{tS}x-\xi\right|^2\right)
\end{align}
is a heat kernel of the perturbed Ornstein-Uhlenbeck operator
\begin{align}
  \label{equ:Linfty2}
  \left[\L_{\infty}v\right](x):=A\triangle v(x)+\left\langle Sx,\nabla v(x)\right\rangle-Bv(x).
\end{align}
Under the same assumptions it is proved in \cite[Theorem 5.3]{Otten2014a} that the family of mappings $T(t):L^p(\R^d,\C^N)\rightarrow L^p(\R^d,\C^N)$, $t\geqslant 0$, 
defined by
\begin{align}
  \left[T(t)v\right](x):= \begin{cases}
                              \int_{\R^d}H(x,\xi,t)v(\xi)d\xi &\text{, }t>0 \\
                              v(x) &\text{, }t=0
                            \end{cases}\quad ,x\in\R^d,
  \label{equ:OrnsteinUhlenbeckSemigroupLp}
\end{align}
generates a strongly continuous semigroup on $L^p(\R^d,\C^N)$ for each $1\leqslant p<\infty$. The semigroup $\left(T(t)\right)_{t\geqslant 0}$ 
is called the Ornstein-Uhlenbeck semigroup if $B=0$. The strong continuity of the semigroup justifies to introduce the infinitesimal generator 
$A_p:L^p(\R^d,\C^N)\supseteq\D(A_p)\rightarrow L^p(\R^d,\C^N)$ of $\left(T(t)\right)_{t\geqslant 0}$, short $\left(A_p,\D(A_p)\right)$, via
\begin{align*}
  A_p v := \lim_{t\downarrow 0}\frac{T(t)v-v}{t},\; 1\leqslant p<\infty
\end{align*}
for every $v\in\D(A_p)$, where the domain (or maximal domain) of $A_p$ is given by
\begin{align*}
  \D(A_p):=&\left\{v\in L^p(\R^d,\C^N)\mid \lim_{t\downarrow 0}\frac{T(t)v-v}{t}\text{ exists in $L^p(\R^d,\C^N)$}\right\} \\
          =&\left\{v\in L^p(\R^d,\C^N)\mid A_p v\in L^p(\R^d,\C^N)\right\}.
\end{align*}
An application of abstract semigroup theory yields the unique solvability of the resolvent equation 
\begin{align*}
  \left(\lambda I-A_p\right)v = g,\quad g\in L^p(\R^d,\C^N),\;\lambda\in\C,\;\lambda>-\max_{\lambda\in\sigma(-B)}\Re\lambda
\end{align*}
in $L^p(\R^d,\C^N)$ for $1\leqslant p<\infty$, \cite[Corollary 5.5]{Otten2014a}, \cite[Corollary 6.7]{Otten2014}. But so far, we neither have any 
explicit representation for the maximal domain $\D(A_p)$ nor do we know anything about the relation between the generator $A_p$ and the differential 
operator $\L_{\infty}$. For this purpose, one has to solve the identification problem, which has been done in \cite{Otten2015a}. 
Assuming \eqref{cond:A8B}, \eqref{cond:A2} and \eqref{cond:A5} for $\K=\C$, it is proved in \cite[Theorem 3.2]{Otten2015a} that the Schwartz space 
$\S(\R^d,\C^N)$ is a core of the infinitesimal generator $\left(A_p,\D(A_p)\right)$ for any $1\leqslant p<\infty$. Next, one considers the operator 
$\L_{\infty}:L^p(\R^d,\C^N)\supseteq\D^p_{\mathrm{loc}}(\L_0)\rightarrow L^p(\R^d,\C^N)$ on its domain 
\begin{align*}
  \D^p_{\mathrm{loc}}(\L_0):=\left\{v\in W^{2,p}_{\mathrm{loc}}(\R^d,\C^N)\cap L^p(\R^d,\C^N)\mid A\triangle v+\left\langle S\cdot,\nabla v\right\rangle\in L^p(\R^d,\C^N)\right\}.
\end{align*} 
Under the assumption \eqref{cond:A3} for $\K=\C$, it is shown in \cite[Lemma 4.1]{Otten2015a} that $\left(\L_{\infty},\D^p_{\mathrm{loc}}(\L_0)\right)$ 
is a closed operator in $L^p(\R^d,\C^N)$ for any $1<p<\infty$, which justifies to introduce and analyze its resolvent. The $L^p$-dissipativity 
condition \eqref{cond:A4DC} is the key assumption which allows an energy estimate with respect to the $L^p$-norm and leads to the following result, 
see \cite[Theorem 4.4]{Otten2015a}.

\begin{theorem}[Resolvent Estimates for $\L_{\infty}$ in $L^p(\R^d,\C^N)$ with $1<p<\infty$]\label{thm:UniquenessInDpmax}
  Let the assumptions \eqref{cond:A4DC} and \eqref{cond:A5} be satisfied for $1<p<\infty$ and $\K=\C$. Moreover, let $\lambda\in\C$ with 
  $\Re\lambda>\beta_B$, where $\beta_B\in\R$ is from \eqref{equ:betaB}, and let $v_{\star}\in\D^p_{\mathrm{loc}}(\L_0)$ denote a solution of
  \begin{align*}
    \left(\lambda I-\L_{\infty}\right)v=g
  \end{align*}
  in $L^p(\R^d,\C^N)$ for some $g\in L^p(\R^d,\C^N)$. Then $v_{\star}$ is the unique solution in $\D^p_{\mathrm{loc}}(\L_0)$ and satisfies the resolvent estimate
  \begin{align*}
    \left\|v_{\star}\right\|_{L^p(\R^d,\C^N)}\leqslant\frac{1}{\Re\lambda-\beta_B}\left\|g\right\|_{L^p(\R^d,\C^N)}.
  \end{align*}
  In addition, for $1<p\leqslant 2$ the following gradient estimate holds
  \begin{align*}
    \left|v_{\star}\right|_{W^{1,p}(\R^d,\C^N)}\leqslant \frac{d^{\frac{1}{p}}\gamma_A^{-\frac{1}{2}}}{\left(\Re\lambda-\beta_B\right)^{\frac{1}{2}}}\left\|g\right\|_{L^p(\R^d,\C^N)}.
  \end{align*} 
\end{theorem}

A direct consequence of Theorem \ref{thm:UniquenessInDpmax} is that the operator $\L_{\infty}$ is dissipative in $L^p(\R^d,\C^N)$ for $1<p<\infty$, provided 
that $\beta_B$ from \eqref{equ:betaB} satisfies $\beta_B\leqslant 0$, \cite[Corollary 4.6]{Otten2015a}. 
Combining Theorem \ref{thm:UniquenessInDpmax}, \cite[Lemma 4.1]{Otten2015a}, \cite[Corollary 5.5]{Otten2014a} and \cite[Theorem 3.2]{Otten2015a} one can solve 
the identification problem for $\L_{\infty}$, which has been done in \cite[Theorem 5.1]{Otten2015a}.

\begin{theorem}[Maximal domain, local version]\label{thm:LpMaximalDomainPart1}
  Let the assumptions \eqref{cond:A8B}, \eqref{cond:A4DC} and \eqref{cond:A5} be satisfied for $1<p<\infty$ and $\K=\C$, then 
  \begin{align*}
    \D(A_p)=\D^p_{\mathrm{loc}}(\L_0)
  \end{align*} 
  is the maximal domain of $A_p$, where $\D^p_{\mathrm{loc}}(\L_0)$ is defined by
  \begin{align}
    \label{equ:localDomain}
    \D^p_{\mathrm{loc}}(\L_0):=\left\{v\in W^{2,p}_{\mathrm{loc}}(\R^d,\C^N)\cap L^p(\R^d,\C^N)\mid A\triangle v+\left\langle S\cdot,\nabla v\right\rangle\in L^p(\R^d,\C^N)\right\}.
  \end{align} 
  In particular, $A_p$ is the maximal realization of $\L_{\infty}$ in $L^p(\R^d,\C^N)$, i.e. $A_p v = \L_{\infty} v$ for every $v\in\D(A_p)$.
\end{theorem}

Theorem \ref{thm:LpMaximalDomainPart1} shows that the $L^p$-dissipativity condition \eqref{cond:A4DC} is crucial to solve the identification problem 
for perturbed complex-valued Ornstein-Uhlenbeck operators. To apply Theorem \ref{thm:LpMaximalDomainPart1} it is helpful to understand which classes 
of matrices $A$ satisfy the algebraic condition \eqref{cond:A4DC}. This motivates to analyze the $L^p$-dissipativity condition \eqref{cond:A4DC} in detail.

% Detailed outline
In Section \ref{sec:TheLpAntieigenvalueCondition} we derive an algebraic characterization of the $L^p$-dissipativity condition \eqref{cond:A4DC} 
in terms of the antieigenvalues of the diffusion matrix $A$. For matrices $A\in\K^{N,N}$ with $\K\in\{\R,\C\}$ we prove in 
Theorem \ref{thm:LpDissipativityCondition} that the $L^p$-dissipativity condition \eqref{cond:A4DC} is satisfied if and only if the 
$L^p$-antieigenvalue condition \eqref{cond:A4} holds. The proof uses the method of Lagrange multipliers, first w.r.t. the $z$-component, then 
w.r.t. the $w$-component. 

In Section \ref{sec:SomeExplicitRepresentationsOfTheFirstAntieigenvalue} we discuss several special cases in which the first antieigenvalue can 
be given explicitly. For Hermitian positive definite matrices $A$ and for normal accretive matrices $A$ we specify well known explicit 
expressions for $\mu_1(A)$ in terms of the eigenvalues of $A$. These representations are proved in \cite[7.4.P4]{HornJohnson2013} for Hermitian 
positive definite matrices and in \cite[Theorem 5.1]{GustafsonSeddighin1989} for normal accretive matrices.

%  ----------------
% | Acknowledgment |
%  ----------------
%\noindent
\textbf{Acknowledgment.} The author is greatly indebted to Wolf-J\"urgen Beyn for extensive discussions which helped in clarifying proofs.

%---------------------------------------------------------------------------------------------------------------------------------------------------
%
%  SECTION 3: ($L^p$-antieigenvalue condition versus $L^p$-dissipativity condition)
%
%---------------------------------------------------------------------------------------------------------------------------------------------------
\sect{\texorpdfstring{$L^p$}{Lp}-dissipativity condition versus \texorpdfstring{$L^p$}{Lp}-antieigenvalue condition}
\label{sec:TheLpAntieigenvalueCondition}
%---------------------------------------------------------------------------------------------------------------------------------------------------

In this section we derive an algebraic characterization of the $L^p$-dissipativity condition \eqref{cond:A4DC} for the perturbed 
complex-valued Ornstein-Uhlenbeck operator $\L_{\infty}$ in $L^p(\R^d,\C^N)$ for $1<p<\infty$. More precisely, the next theorem shows 
that the $L^p$-dissipativity condition \eqref{cond:A4DC} is equivalent to a lower bound for the first antieigenvalue of the diffusion 
matrix $A$. The proof is based on an application of the method of Lagrange multipliers. An application of Theorem \ref{thm:LpDissipativityCondition} 
to $b:=p-2$ for $1<p<\infty$ shows that \eqref{cond:A4DC} and \eqref{cond:A4} are equivalent. The equivalence allows us to require either 
\eqref{cond:A4DC} or \eqref{cond:A4} in Theorem \ref{thm:UniquenessInDpmax} and Theorem \ref{thm:LpMaximalDomainPart1}.

%  -------------
% | Theorem 3.1 |
%  -------------
\begin{theorem}[$L^p$-dissipativity condition vs. $L^p$-antieigenvalue condition]\label{thm:LpDissipativityCondition}
  Let $A\in\K^{N,N}$ for $K=\R$ if $N\geqslant 2$ and $\K=\C$ if $N\geqslant 1$, and let $b\in\R$ with $b>-1$. \\ 
  \enum{a} Given some $\gamma_A>0$, then the following statements are equivalent:
  \begin{align}
    &|z|^2\Re\left\langle w,Aw\right\rangle + b\Re\left\langle w,z\right\rangle\Re\left\langle z,Aw\right\rangle \geqslant \gamma_A|z|^2|w|^2 &&\forall\,w,z\in\K^N, \label{equ:DCProp1}\\
    &\left(1+\frac{b}{2}\right)\Re\left\langle w,Aw\right\rangle - \frac{|b|}{2}\left|Aw\right| \geqslant \gamma_A                            &&\forall\,w\in\K^N,\,|w|=1. \label{equ:DCProp2}
  \end{align}
  \enum{b} Moreover, the following statements are equivalent:
  \begin{align}
    &\exists\,\gamma_A>0:\;\left(1+\frac{b}{2}\right)\Re\left\langle w,Aw\right\rangle-\frac{|b|}{2}\left|Aw\right|\geqslant\gamma_A &&\forall\,w\in\K^N,\,|w|=1, \label{equ:DCProp3}\\
    &A\text{ invertible}\quad\text{and}\quad\mu_1(A) > \frac{|b|}{2+b}, \label{equ:DCProp4}
  \end{align}
  where $\mu_1(A)$ denotes the first antieigenvalue of $A$ in the sense of Definition \ref{def:AntieigenvalueAntieigenvectorAngle}.
\end{theorem}

%  ---------------------
% | Proof (Theorem 3.1) |
%  ---------------------
\begin{proof}
  \enum{a}: Obviously, dividing both sides by $|z|^2|w|^2$, \eqref{equ:DCProp1} is equivalent to
  \begin{align}
    \Re\left\langle w,Aw\right\rangle + b\Re\left\langle w,z\right\rangle\Re\left\langle z,Aw\right\rangle \geqslant \gamma_A &&\forall\,w,z\in\K^N,\,|z|=|w|=1. \label{equ:DCProp5}
  \end{align}
  We now prove the equivalence of \eqref{equ:DCProp5} and \eqref{equ:DCProp2}. The case $b=0$ is trivial, so assume w.l.o.g. $b\neq 0$. 
  We distinguish between the cases $\K=\R$ and $\K=\C$. \\
  \textbf{Case 1:} ($\K=\R$). Let $N\geqslant 2$. In this case we show the equivalence of
  \begin{align}
    &\left\langle w,Aw\right\rangle + b\left\langle w,z\right\rangle\left\langle z,Aw\right\rangle \geqslant \gamma_A 
      &&\forall\,w,z\in\R^N,\,|z|=|w|=1, \label{equ:Lpdissipativity1} \\
    &\left(1+\frac{b}{2}\right)\left\langle w,Aw\right\rangle - \frac{|b|}{2}\left|Aw\right| \geqslant \gamma_A
      &&\forall\,w\in\R^N,\,|w|=1, \label{equ:Lpdissipativity2}
  \end{align}
  for some $\gamma_A>0$ by minimizing \eqref{equ:Lpdissipativity1} with respect to $z$ subject to $|z|^2=1$. Note that the minimum 
  exists due to the boundedness of
  \begin{align*}
    \left|\left\langle z,Aw\right\rangle\left\langle w,z\right\rangle\right| \leqslant |z|^2|Aw||w| = |Aw|.
  \end{align*}
  \textbf{Subcase 1:} ($w$, $Aw$ linearly dependent). Let $w$ and $Aw$ be linearly dependent, then there exists $\lambda\in\R$ 
  such that $Aw=\lambda w$. Since $|w|=1$, we conclude $w\neq 0$ and therefore, $\lambda\in\sigma(A)$. Applying \eqref{equ:Lpdissipativity1} with $z:=w$
  \begin{align*}
    0<\gamma_A\leqslant \left\langle w,Aw\right\rangle + b\left\langle w,w\right\rangle\left\langle w,Aw\right\rangle = (1+b)\lambda
  \end{align*}
  we deduce $\lambda>0$, since $b>-1$. In this case \eqref{equ:Lpdissipativity1} and \eqref{equ:Lpdissipativity2} reads as
  \begin{align}
    &\lambda\left|w\right|^2 + \lambda b \left\langle w,z\right\rangle^2 \geqslant \gamma_A 
      &&\forall\,w,z\in\R^N,\,|z|=|w|=1, \label{equ:Lpdissipativity1REAL} \\
    &\left(1+\frac{b}{2}\right)\lambda\left|w\right|^2 - \frac{|b|}{2}\left|\lambda\right|\left|w\right| \geqslant \gamma_A
      &&\forall\,w\in\R^N,\,|w|=1. \label{equ:Lpdissipativity2REAL}
  \end{align}
  The aim follows by minimization of $\lambda b\left\langle w,z\right\rangle^2$ with respect to $z$ subject to $|z|^2=1$. If $b>0$ then $\lambda b>0$ 
  and therefore, $\lambda b\left\langle w,z\right\rangle^2$ is minimal iff $\left\langle w,z\right\rangle^2$ is minimal. Choose $z\in w^{\perp}$ with 
  $|z|=1$ then the minimum is
  \begin{align*}
    \min_{\substack{z\in\R^N\\|z|=1}}\lambda b\left\langle w,z\right\rangle^2 = \min_{\substack{z\in w^{\perp}\\|z|=1}}\lambda b\left\langle w,z\right\rangle^2 = 0.
  \end{align*}
  If $b<0$ then $\lambda b<0$ and therefore, $\lambda b\left\langle w,z\right\rangle^2$ is minimal iff $\left\langle w,z\right\rangle^2$ is maximal. 
  Choose $z\in\left\{w,-w\right\}$ then the minimum is
  \begin{align*}
    \min_{\substack{z\in\R^N\\|z|=1}}\lambda b\left\langle w,z\right\rangle^2 = \lambda b<0.
  \end{align*}
  \textbf{Subcase 2:} ($w$, $Aw$ linearly independent). For this purpose we use the method of Lagrange multipliers for 
  finding the local minima of \eqref{equ:Lpdissipativity1} w.r.t. $z$. Consider the functions
  \begin{align*}
    f(w,z):=&\left\langle w,Aw\right\rangle + b\left\langle w,z\right\rangle\left\langle z,Aw\right\rangle-\gamma_A, \\
      g(z):=&|z|^2-1=0
  \end{align*}
  for every fixed $w\in\R^N$ with $|w|=1$. The optimization problem is to minimize $f(w,z)$ w.r.t. $z\in\R^N$ subject to the constraint $g(z)=0$. \\
  1. We introduce a new variable $\mu\in\R$, called the Lagrange multiplier, and define the Lagrange function (Lagrangian)
  \begin{align*}
    \Lambda:\R^N\times\R\rightarrow\R,\quad \Lambda(z,\mu):=f(z,w)+\mu g(z).
  \end{align*}
  The solution of the minimization problem corresponds to a critical point of the Lagrange function. A necessary condition for critical points of $\Lambda$ is that 
  the Jacobian vanishes, i.e. $J_{\Lambda}(z,\mu)=0$. This leads to the equations
  \begin{align}
    &b\left\langle z,Aw\right\rangle w+b\left\langle w,z\right\rangle Aw + 2\mu z=0, \label{equ:Lpdissipativity3} \\
    &|z|^2-1 = 0, \label{equ:Lpdissipativity4}
  \end{align}
  i.e. every local minimizer $z$ satisfies \eqref{equ:Lpdissipativity3} and \eqref{equ:Lpdissipativity4}. \\
  2. Multiplying \eqref{equ:Lpdissipativity3} from the left by $z^T$ and using \eqref{equ:Lpdissipativity4} we obtain
  \begin{align*}
    0 = 2b\left\langle z,Aw\right\rangle\left\langle w,z\right\rangle + 2\mu |z|^2 = 2b\alpha\beta + 2\mu,
  \end{align*}
  and thus $\mu=-b\alpha\beta$, where $\alpha:=\left\langle z,Aw\right\rangle$ and $\beta:=\left\langle w,z\right\rangle$ are still to be determined. Now, inserting 
  $\mu=-b\alpha\beta$ into \eqref{equ:Lpdissipativity3} and dividing both sides by $b\neq 0$ we obtain
  \begin{align}
    \alpha w+\beta Aw -2\alpha\beta z = 0. \label{equ:Lpdissipativity5}
  \end{align}
  From \eqref{equ:Lpdissipativity5} we deduce that if $\alpha=0$ then $\beta=0$ and vice versa. If $\alpha=\beta=0$ then $z\in\left\{w,Aw\right\}^{\perp}$ and the 
  minimum of $f(w,z)$ in $z$ subject to $g(z)=0$ is $\left\langle w,Aw\right\rangle-\gamma_A$.\\
  In the following we consider the case $\alpha\neq 0$ and $\beta\neq 0$ and we show that in this case the minimum of $f(w,z)$ in $z$ subject to $g(z)=0$ is even smaller. 
  Note that, assuming $\alpha\neq 0$ and $\beta\neq 0$, \eqref{equ:Lpdissipativity5} yields the following representation for $z$
  \begin{align}
    \label{equ:RepresentationForMinimizerZ}
    z = \frac{1}{2\alpha\beta}\left(\alpha w + \beta Aw\right) = \frac{1}{2\beta}w + \frac{1}{2\alpha}Aw,
  \end{align}
  We now look for possible solutions for $\alpha$ and $\beta$. \\
  3. Multiplying \eqref{equ:Lpdissipativity5} from the left by $w^T$ and using $|w|=1$ we obtain
  \begin{align}
    0 = \alpha |w|^2 + \beta\left\langle w,Aw\right\rangle - 2\alpha\beta\left\langle w,z\right\rangle
      = \alpha + \beta q - 2\alpha\beta^2, \label{equ:AlphaBetaEquation1}
  \end{align}
  where $q:=\left\langle w,Aw\right\rangle$. Multiplying \eqref{equ:Lpdissipativity5} from the left by $(Aw)^T$ we obtain
  \begin{align}
    0 = \alpha\beta\left\langle Aw,w\right\rangle + \beta\left\langle Aw,Aw\right\rangle -2\alpha\beta\left\langle Aw,z\right\rangle
      = \alpha q + \beta r^2 - 2\alpha^2\beta, \label{equ:AlphaBetaEquation2}
  \end{align}
  where $r:=|Aw|$. From \eqref{equ:Lpdissipativity1} with $z:=w$ we deduce that $q>0$ since $b>-1$ and $\gamma_A>0$. Moreover, we have $r>0$: Assuming $r=|Aw|=0$ yields $Aw=0$ for 
  some $|w|=1$ which contradicts $\gamma_A>0$, compare \eqref{equ:Lpdissipativity1}. Since $r>0$, $q>0$ and by assumption $\alpha\neq 0$ and $\beta\neq 0$, 
  there exist four solutions of \eqref{equ:AlphaBetaEquation1}, \eqref{equ:AlphaBetaEquation2} given by
  \begin{align}
    \label{equ:SolutionsForAlphaBeta}
    (\alpha,\beta)\in\left\{\left(\mp\sqrt{\frac{r(r-q)}{2}},\pm\sqrt{\frac{r-q}{2r}}\right),\left(\pm\sqrt{\frac{r(r+q)}{2}},\pm\sqrt{\frac{r+q}{2r}}\right)\right\}.
  \end{align}
  Note, that $r\pm q>0$ and therefore $\left(\alpha,\beta\right)\neq (0,0)$. This follows from the Cauchy-Schwarz inequality and $|w|=1$
  \begin{align*}
    \pm q\leqslant |q| = \left|\left\langle w,Aw\right\rangle\right| < |w||Aw| = r.
  \end{align*}
  Note that we have indeed a strict inequality since $w$ and $Aw$ are linearly independent by our subcase. \\
  4. Instead of investigating whether the Hessian of $f$ at these points is positive definite or not, we evaluate the function $f$ at the points \eqref{equ:RepresentationForMinimizerZ} 
  with $(\alpha,\beta)$ from \eqref{equ:SolutionsForAlphaBeta} directly. First we observe that
  \begin{align}
    \label{equ:RepresentationFofWAndZ}
    f(w,z) = \left\langle w,Aw\right\rangle + b\left\langle w,z\right\rangle\left\langle z,Aw\right\rangle-\gamma_A
           = q + b\alpha\beta - \gamma_A.
  \end{align}
  We now distinguish between the two cases $b>0$ and $b<0$. If $b>0$ then the function $f(w,z)$ is minimal if $\mathrm{sgn}\,\alpha=-\mathrm{sgn}\,\beta$ and if 
  $b<0$ then $f(w,z)$ is minimal if $\mathrm{sgn}\,\alpha=\mathrm{sgn}\,\beta$. Therefore, for the choice of
  \begin{align}
    \label{equ:AlphaBetaSolutions}
    (\alpha,\beta)=\begin{cases}
                     \left(\mp\sqrt{\frac{r(r-q)}{2}},\pm\sqrt{\frac{r-q}{2r}}\right) &,\,b>0, \\
                     \left(\pm\sqrt{\frac{r(r+q)}{2}},\pm\sqrt{\frac{r+q}{2r}}\right) &,\,b<0,
                   \end{cases}
  \end{align}
  the term $b\alpha\beta$ is negative and we have found the global minimum. Thus, for $b>0$ we obtain
  \begin{align}
    \label{equ:CaseBpositive}
    b\alpha\beta = -b\sqrt{\frac{r(r-q)}{2}}\sqrt{\frac{r-q}{2r}} = -\frac{b}{2}(r-q) = -\frac{|b|}{2}r+\frac{b}{2}q < 0 
  \end{align}
  and similarly for $b<0$ we obtain
  \begin{align}
    \label{equ:CaseBnegative}
    b\alpha\beta = b\sqrt{\frac{r(r+q)}{2}}\sqrt{\frac{r+q}{2r}} = \frac{b}{2}(r+q) = -\frac{|b|}{2}r+\frac{b}{2}q < 0.
  \end{align}
  Therefore, using \eqref{equ:RepresentationFofWAndZ}, \eqref{equ:CaseBpositive} and \eqref{equ:CaseBnegative}, the global minimum of $f(w,z)$ in $z$ subject 
  to the constraint $g(z)=0$ is given by
  \begin{align*}
    \min_{\substack{z\in\R^N\\|z|=1}}f(w,z) = \min_{\substack{z\in\R^N\\|z|=1}}\left(q+b\alpha\beta-\gamma_A\right) = \left(1+\frac{b}{2}\right)q - \frac{|b|}{2}r - \gamma_A
  \end{align*}
  for every fixed $w\in\R^N$ with $|w|=1$. In particular, defining 
  \begin{align}
    \label{equ:MinimizerZ}
    (z_{\star},\mu_{\star}) = \left(\frac{1}{2\beta}w+\frac{1}{2\alpha}Aw,-b\alpha\beta\right)\text{ with $\alpha$, $\beta$ from \eqref{equ:AlphaBetaSolutions}}.
  \end{align}
  the above minimum is attained at $z_{\star}$ from \eqref{equ:MinimizerZ} since
  \begin{align}
    \label{equ:EquationForW}
    f(w):=f(w,z_{\star})=\left(1+\frac{b}{2}\right)q - \frac{|b|}{2}r - \gamma_A
  \end{align}
  for every fixed $w\in\R^N$ with $|w|=1$. Taking \eqref{equ:DCProp5} into account, \eqref{equ:EquationForW} must be nonnegative for every $w\in\R^N$ with $|w|=1$. 
  This corresponds exactly \eqref{equ:DCProp2}. \\
  \textbf{Case 2:} ($\K=\C$). In this case we apply Case 1 (with $\K=\R$). For this purpose, we write
  \begin{align*}
    \C^N \ni w &= w_1+iw_2 \cong \begin{pmatrix}w_1\\w_2\end{pmatrix} = w_{\R} \in\R^{2N}, \\
    \C^N \ni z &= z_1+iz_2 \cong \begin{pmatrix}z_1\\z_2\end{pmatrix} = z_{\R} \in\R^{2N}, \\
    \C^{N,N} \ni A &= A_1+iA_2 \cong \begin{pmatrix}A_1 &-A_2\\A_2 &A_1\end{pmatrix} = A_{\R} \in\R^{2N,2N}.
  \end{align*}
  From
  \begin{align*}
    \left\langle w,z\right\rangle &= \left\langle w_1,z_1\right\rangle + \left\langle w_2,z_2\right\rangle + i\left(\left\langle w_1,z_2\right\rangle - \left\langle w_2,z_1\right\rangle\right)
  \end{align*}
  we deduce
  \begin{align*}
    \Re\left\langle w,z\right\rangle = \left\langle w_{\R},z_{\R}\right\rangle,\quad
    \Re\left\langle w,Aw\right\rangle = \left\langle w_{\R},A_{\R}w_{\R}\right\rangle,\quad
    \left|Aw\right| = \left|A_{\R}w_{\R}\right|.
  \end{align*}
  Therefore, \eqref{equ:DCProp5} translates into
  \begin{align*}
    \left\langle w_{\R},A_{\R}w_{\R}\right\rangle + b\left\langle w_{\R},z_{\R}\right\rangle\left\langle z_{\R},A_{\R}w_{\R}\right\rangle \geqslant \gamma_A 
      &&\forall\,w_{\R},z_{\R}\in\R^{2N},\,|z_{\R}|=|w_{\R}|=1.
  \end{align*}
  Due to Case 1 this is equivalent to
  \begin{align*}
    \left(1+\frac{b}{2}\right)\left\langle w_{\R},A_{\R}w_{\R}\right\rangle - \frac{|b|}{2}\left|A_{\R}w_{\R}\right| \geqslant \gamma_A 
      &&\forall\,w_{\R}\in\R^{2N},\,|w_{\R}|=1,
  \end{align*}
  that translates back into
  \begin{align*}
    \left(1+\frac{b}{2}\right)\Re\left\langle w,Aw\right\rangle - \frac{|b|}{2}\left|Aw\right| \geqslant \gamma_A
      &&\forall\,w\in\C^N,\,|w|=1,
  \end{align*}
  which proves the case $\K=\C$. \\
  \enum{b}: We prove that \eqref{equ:DCProp3} is equivalent to
  \begin{align}
    \text{$A$ is invertible}\quad\text{and}\quad\exists\,\delta_A>1:\;\frac{\left(2+b\right)}{|b|}\cdot\frac{\Re\left\langle w,Aw\right\rangle}{|w||Aw|}\geqslant \delta_A\;\forall\,w\in\K^N,\,w\neq 0,\,Aw\neq 0. \label{equ:DCProp6}
  \end{align}
  Then, by Definition \ref{def:AntieigenvalueAntieigenvectorAngle} of the first antieigenvalue \eqref{equ:DCProp6} is equivalent to
  \begin{align}
    \text{$A$ is invertible}\quad\text{and}\quad\exists\,\delta_A>1:\;\frac{\left(2+b\right)}{|b|}\cdot\mu_1(A)\geqslant \delta_A. \label{equ:DCProp7}
  \end{align}
  and, obviously, \eqref{equ:DCProp7} is equivalent to \eqref{equ:DCProp4}. This completes the proof.\\
  \eqref{equ:DCProp3}$\Longleftarrow$\eqref{equ:DCProp6}: Multiplying the numerator and the denominator by $\frac{1}{|w|^2}$, allows us to consider \eqref{equ:DCProp6} 
  for $w\in\K^N$ with $|w|=1$ and $Aw\neq 0$. Since $A$ is invertible, $Aw\neq 0$ is satisfied for every $w\in\K^N$ with $|w|=1$. Multiplying \eqref{equ:DCProp6} by 
  $\frac{|b|}{2}|w||Aw|$ and using the inequality $|w|=|A^{-1}Aw|\leqslant |A^{-1}||Aw|$ we obtain 
  \begin{align*}
    \left(1+\frac{b}{2}\right)\Re\left\langle w,Aw\right\rangle &\geqslant \frac{|b|}{2}|w||Aw|\delta_A = \frac{|b|}{2}|w||Aw|+\frac{|b|}{2}|w||Aw|\left(\delta_A-1\right) \\
    &\geqslant \frac{|b|}{2}|w||Aw|+\frac{|b|}{2}\frac{|w|}{|A^{-1}|}\left(\delta_A-1\right) = \frac{|b|}{2}|Aw|+\gamma_A,
  \end{align*}
  for every $w\in\K^N$ with $|w|=1$, where $\gamma_A:=\frac{|b|}{2}\frac{1}{|A^{-1}|}\left(\delta_A-1\right)$.\\
  \eqref{equ:DCProp3}$\Longrightarrow$\eqref{equ:DCProp6}: Let $\lambda_j^A$ for $j=1,\ldots,N$ denote the $j$-th eigenvalue corresponding to the $j$-th eigenvector 
  $v_j$ with $|v_j|=1$ of the matrix $A$. Then the multiplication of \eqref{equ:DCProp3} by $\frac{2}{2+b}$ implies
  \begin{align*}
    \Re\lambda_j^A = \Re\lambda_j^A |v_j|^2 = \Re\left\langle v_j,Av_j\right\rangle 
    \geqslant \Re\left\langle v_j,Av_j\right\rangle -\frac{|b|}{2+b}|Av_j| \geqslant \gamma_A\frac{2}{2+b} >0,
  \end{align*}
  thus $\Re\sigma(A)>0$ and hence, $A$ is invertible. Multiplying \eqref{equ:DCProp3} by $\frac{2}{|b||Aw|}$ we obtain
  \begin{align*}
    \frac{(2+b)}{|b|}\frac{\Re\left\langle w,Aw\right\rangle}{|Aw|}\geqslant \frac{2}{|b|}\frac{\gamma_A}{|Aw|}+1\quad\forall\,w\in\K^N,\,|w|=1,\,Aw\neq 0.
  \end{align*}
  Now, let $w\in\K^N$ with $w\neq 0$ and $Aw\neq 0$, then $\left|\frac{w}{|w|}\right|=1$ and we further obtain
  \begin{align*}
    \frac{(2+b)}{|b|}\frac{\Re\left\langle w,Aw\right\rangle}{|w||Aw|}
    \geqslant \frac{2}{|b|}\frac{\gamma_A}{|A|}+1=:\delta_A>1\quad \forall\,w\in\K^N,\,w\neq 0,\,Aw\neq 0,
  \end{align*}
  where we used $|Aw|\leqslant |A||w|$.
\end{proof}

%---------------------------------------------------------------------------------------------------------------------------------------------------
%
%  SECTION 4: (Special cases and explicit representations of the first antieigenvalue)
%
%---------------------------------------------------------------------------------------------------------------------------------------------------
\sect{Special cases and explicit representations of the first antieigenvalue}
\label{sec:SomeExplicitRepresentationsOfTheFirstAntieigenvalue}
%---------------------------------------------------------------------------------------------------------------------------------------------------
 
An application of Theorem \ref{thm:LpDissipativityCondition} with $b:=p-2$ and $1<p<\infty$ implies that the $L^p$-dissipativity condition \eqref{cond:A4DC} 
is equivalent to our new $L^p$-antieigenvalue condition \eqref{cond:A4} which states that the diffusion matrix $A$ is invertible and satisfies the 
$L^p$-antieigenvalue bound
\begin{align*}
  \mu_1(A) > \frac{|p-2|}{p}\in[0,1[,\quad 1<p<\infty.
\end{align*}
This lower $p$-dependent bound for the first antieigenvalue of $A$ is equivalent to an upper $p$-dependent bound for the (real) angle of $A$
\begin{align*}
  \Phi_{\R}(A):=\cos^{-1}\left(\mu_1(A)\right)<\cos^{-1}\left(\frac{|p-2|}{p}\right)\in\,]0,\frac{\pi}{2}],\quad 1<p<\infty.
\end{align*}
In this section we discuss several special cases in which the first antieigenvalue of the matrix $A$ can be given explicitly. In addition, we analyze the 
geometric meaning of the $L^p$-antieigenvalue bound and investigate its behavior for $1<p<\infty$. Note that for general matrices $A$ one cannot expect 
that there is an explicit expression for the first antieigenvalue $\mu_1(A)$ of a matrix $A$. However, for certain classes of matrices it is possible to 
derive a closed formula for $\mu_1(A)$ as it is shown in the following. These explicit representations facilitate to check the validity of the $L^p$-antieigenvalue 
bound.

% 4.1. The scalar real case: (Positivity).
\subsection{The scalar real case: (Positivity)}\label{sec:4.1}
In the scalar real case $A=a\in\R$ (with $\K=\R$ and $N=1$) the statements \eqref{equ:DCProp1} and \eqref{equ:DCProp5} are equivalent, but they are in general 
not equivalent with \eqref{equ:DCProp2}. In particular, there exists a constant $\gamma_{a}$ with \eqref{equ:DCProp5} if and only if 
$(p-1)a=(1+b)a>0$. Since $b=p-2$ with $1<p<\infty$, this is equivalent to $a>0$, compare assumption \eqref{cond:A4}. Note that the scalar real case has not been 
treated in Theorem \ref{thm:LpDissipativityCondition} and therefore, it has been analyzed here. We point out that in this case the first antieigenvalue bound 
does not appear.

% 4.2. The scalar complex case: (A cone condition).
\subsection{The scalar complex case: (A cone condition)}\label{sec:4.2}
In the scalar complex case $A=\alpha\in\C$ (with $\K=\C$ and $N=1$) there exists a constant $\gamma_{\alpha}$ with \eqref{equ:DCProp2}, $b:=p-2$ and $1<p<\infty$, 
if and only if one of the following \begriff{cone conditions} hold
\begin{align}
  &\frac{\left|p-2\right|}{2\sqrt{p-1}}\left|\Im\alpha\right|<\Re\alpha,                                                              \label{equ:DCProp8}\\
  &\left|\arg\alpha\right|<\cos^{-1}\left(\frac{\left|p-2\right|}{p}\right)=\arctan\left(\frac{2\sqrt{p-1}}{\left|p\right|}\right). \label{equ:DCProp9}
\end{align}
This conditions will be discussed below for normal matrices in more details. Condition \eqref{equ:DCProp8} has also been established in \cite[Theorem 2]{CialdeaMazya2005} 
for differential operators with constant coefficients and in \cite[Theorem 4]{CialdeaMazya2005} for differential operators with variable coefficients but on 
bounded domains. Therefore, this result can be considered as an extension of \cite{CialdeaMazya2005}. 

% 4.3. \mu_1(A) for Hermitian positive definite matrices
\subsection{\texorpdfstring{$\boldsymbol{\mu_1(A)}$}{mu1(A)} for Hermitian positive definite matrices}\label{sec:4.3}
If $A$ is a Hermitian positive definite matrix, then $\mu_1(A)$ is given by, \cite[7.4.P4]{HornJohnson2013}, 
\begin{align}
  \mu_1(A) = \frac{\sqrt{\lambda_1^A\lambda_N^A}}{\frac{1}{2}\left(\lambda_1^A+\lambda_N^A\right)} = \frac{2\sqrt{\kappa_A}}{\kappa_A+1}
           = \frac{\mathrm{GeometricMean(\lambda_1^A,\lambda_N^A)}}{\mathrm{ArithmeticMean(\lambda_1^A,\lambda_N^A)}}, \label{equ:DCProp10}
\end{align}
where $0<\lambda_1^A\leqslant \lambda_2^A\leqslant\cdots\leqslant\lambda_N^A$ denote the (real) positive eigenvalues of $A$ and $\kappa_A:=\frac{\lambda_N^A}{\lambda_1^A}$ 
denotes the \begriff{spectral condition number of $A$}. In this case $\mu_1(A)$ is the quotient of the geometric mean $\sqrt{\lambda_1^A\lambda_N^A}$ 
and the arithmetic mean $\frac{1}{2}\left(\lambda_1^A+\lambda_N^A\right)$ of the smallest and largest eigenvalue of $A$. In particular, the equality 
$\mu_1(A)=\frac{\Re\left\langle w,Aw\right\rangle}{|Aw|}$ is satisfied for the antieigenvector $w=\sqrt{\lambda_N^A}u_1+\sqrt{\lambda_1^A}u_N$, 
where $u_1,u_N\in\K^N$ are orthogonal vectors with $Au_1=\lambda_1^A u_1$ and $Au_N=\lambda_N^A u_N$ such that $|w|=1$. This follows directly from the Greub-Rheinboldt inequality, 
\cite[(7.4.12.11)]{HornJohnson2013}, and can be found in \cite[7.4.P4]{HornJohnson2013} and \cite[Corollary 2]{Mirman1983}.
 
Note that for $1<p<\infty$ the $L^p$-antieigenvalue condition \eqref{cond:A4} and \eqref{equ:DCProp10} imply 
\begin{align*}
  \frac{\sqrt{\lambda_1^A\lambda_N^A}}{\frac{1}{2}\left(\lambda_1^A+\lambda_N^A\right)} = \mu_1(A) > \frac{|p-2|}{p}
  \quad\Longleftrightarrow\quad \left(\frac{1}{2}-\frac{1}{p}\right)^2\left(\lambda_1^A+\lambda_N^A\right)^2 < \lambda_1^A\lambda_N^A.
  %&\left(\text{i.e. }\frac{1}{2}\left(1-\frac{2}{p}\right)^2\mathrm{ArithmeticMean}^2(\lambda_1^A,\lambda_N^A)<\mathrm{GeometricMean}^2(\lambda_1^A,\lambda_N^A)\right).
\end{align*}
The latter inequality also appears in \cite[Theorem 7]{Cialdea2009}, where the authors analyzed $L^p$-dissipativity of the differential operator $\nabla^T(Q\nabla v)$ 
for symmetric, positive definite matrices $Q\in\R^{d,d}$. 

If we define $q:=\frac{|p-2|}{p}$ for $1<p<\infty$, then $q\in[0,1)$ and the $L^p$-antieigenvalue condition \eqref{cond:A4} with \eqref{equ:DCProp10} is equivalent to
\begin{align*}
  \frac{2-q^2-2\sqrt{1-q^2}}{q^2}<\kappa_A<\frac{2-q^2+2\sqrt{1-q^2}}{q^2},\text{ for $0<q<1$.}
\end{align*}
Using the definition of $q$, this inequality implies
\begin{align*}
  C_L(p):=\frac{p^2+4p-4-4p\sqrt{p-1}}{(p-2)^2}<\kappa_A<\frac{p^2+4p-4+4p\sqrt{p-1}}{(p-2)^2}=:C_R(p),
\end{align*}
for $1<p<\infty$ and $p\neq 2$. These are lower and upper bounds for the spectral condition number of $A$. Of course, since 
$0<\lambda_1^A\leqslant \lambda_2^A\leqslant\cdots\leqslant\lambda_N^A$ not only $\kappa_A=\frac{\lambda_N^A}{\lambda_1^A}$ but also $\frac{\lambda_j^A}{\lambda_1^A}$ must 
be contained in the open interval $(C_L(p),C_R(p))$ for every $1\leqslant j\leqslant N$. The behavior of the constants $C_L(p)$ and $C_R(p)$ is depicted in 
Figure \ref{fig:LpDissipativityHermitianMatrices}(a). In particular, to satisfy this condition for arbitrary large $p$, i.e. $p$ near $\infty$, the matrix $A$ must be 
of the form $A=aI_N$ for some $0<a\in\R$.
%\begin{figure}[ht]
%  \centering
%  \subfigure[]{\includegraphics[width=0.48\textwidth]{images/LpDissipativityQ.eps} \label{fig:BehaviorOfConstantQ} }
%  \subfigure[]{\includegraphics[width=0.48\textwidth]{images/LpDissipativityHermitianMatrix.eps} \label{fig:LpDissipativityHermitian} }
%  \caption{$q$ as a function on $p$ (a) and constants $C_L$ (red) and $C_R$ (blue) in dependence on $p$ (b)}
%  \label{fig:LpDissipativityHermitianMatrices}
%\end{figure}
\begin{figure}[ht]
  \centering
  \subfigure[]{\includegraphics[width=0.48\textwidth]{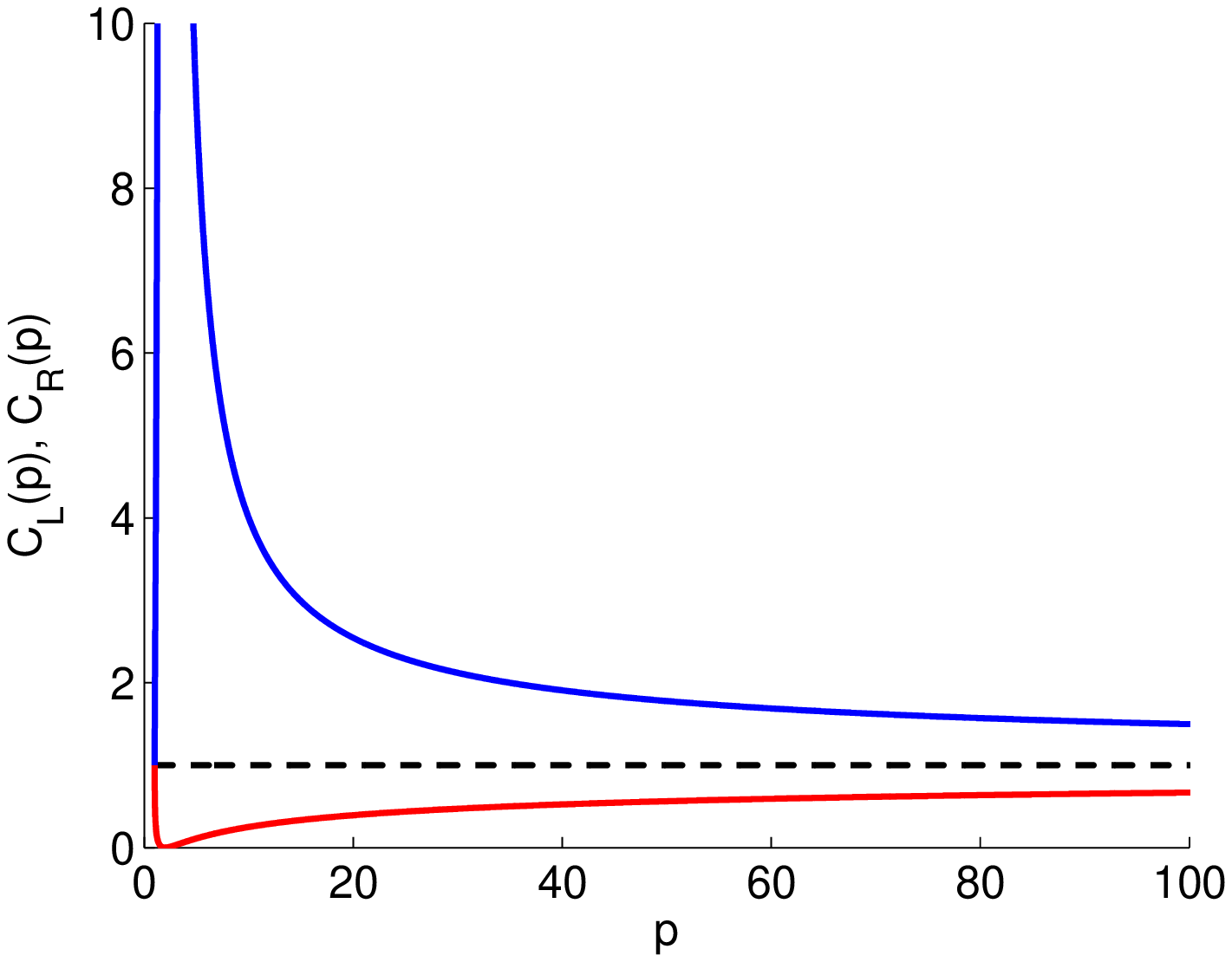} \label{fig:LpDissipativityHermitian} }
  \subfigure[]{\begin{psfrags}
                    \psfrag{ReLambda}[b][b]{\footnotesize$\Re\lambda$}
                    \psfrag{ImLambda}[b][b]{\footnotesize$\Im\lambda$}
                    \psfrag{p1to2}[b][b]{\footnotesize\textcolor{red}{$1<p<2$}}
                    \psfrag{p2toinfinity}[b][b]{\footnotesize\textcolor{blue}{$2<p<\infty$}}
                    \psfrag{phi}[b][b]{\footnotesize\hspace{0.2cm}$\left|\arg\lambda\right|=\cos^{-1}\left(\frac{\left|p-2\right|}{p}\right)$}
                    \psfrag{Sigmap}[b][b]{\footnotesize\textcolor{green}{$\Sigma_p$}}
                    \includegraphics[width=0.48\textwidth]{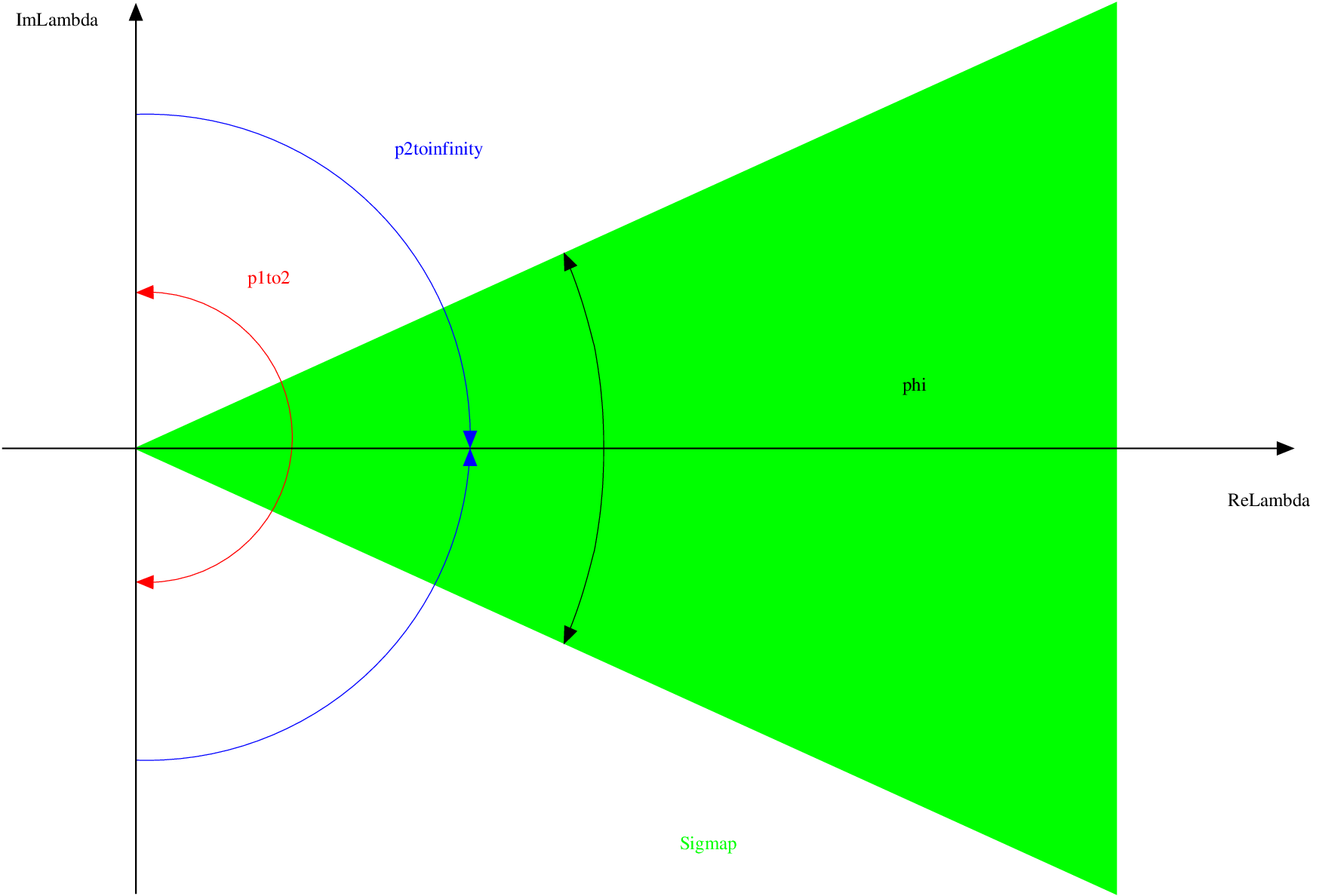}
                    \end{psfrags}
                    \label{fig:SectorForDissipativity}}
  \caption{(a) $p$-dependent bounds $C_L$ (red) and $C_R$ (blue) for the spectral condition number of Hermitian positive definite matrices $A$, (b) conic section for the antieigenvalue assumption \eqref{cond:A4} for normal accretive matrices $A$}
  \label{fig:LpDissipativityHermitianMatrices}
\end{figure}

% 4.4. \mu_1(A) for normal accretive matrices
\subsection{\texorpdfstring{$\boldsymbol{\mu_1(A)}$}{mu1(A)} for normal accretive matrices}\label{sec:4.4}
If $A$ is a normal accretive matrix, then $\mu_1(A)$ from \eqref{equ:DCProp4} is given by $\mu_1(A) = \min(E\cup F)$, where
\begin{align*}
  E :=& \left\{\left.\frac{a_j}{\left|\lambda_j^A\right|}\right| 1\leqslant j\leqslant N\right\},\\
  F :=& \left\{\rule{0cm}{1.1cm}\left.\frac{2\sqrt{\left(a_j-a_i\right)\left(a_i\left|\lambda_j^A\right|^2-a_j\left|\lambda_i^A\right|^2\right)}}{\left|\lambda_j^A\right|^2-\left|\lambda_i^A\right|^2}\right|\right.
          0<\frac{a_j\left|\lambda_j^A\right|^2-2a_i\left|\lambda_j^A\right|^2+a_j\left|\lambda_i^A\right|^2}{\left(\left|\lambda_i^A\right|^2-\left|\lambda_j^A\right|^2\right)
          \left(a_i-a_j\right)}<1,\\
        & \qquad\qquad\qquad\qquad\qquad\qquad\qquad\qquad\;\;\left.\rule{0cm}{1.1cm}1\leqslant i,j\leqslant N,\,\left|\lambda_i^A\right|\neq\left|\lambda_j^A\right|\right\},
\end{align*}
and $\lambda_j^A=a_j+i b_j$ with $a_j,b_j\in\R$, $1\leqslant j\leqslant N$, denote the eigenvalues of $A$. In particular, if 
\begin{align}
  \label{equ:AntieigenvalueNormalMatrixPossibility1}
  \mu_1(A)=\frac{a_j}{\left|\lambda_j^A\right|}\text{ for some $1\leqslant j\leqslant N$,}
\end{align}
then $\mu_1(A)=\frac{\Re\left\langle w,Aw\right\rangle}{|Aw|}$ for an antieigenvector $w\in\K^N$ with $\left|w_j\right|=1$ and $\left|w_k\right|=0$ for $1\leqslant k\leqslant N$ 
with $k\neq j$. Conversely, if 
\begin{align}
  \label{equ:AntieigenvalueNormalMatrixPossibility2}
  \mu_1(A)=\frac{2\sqrt{\left(a_j-a_i\right)\left(a_i\left|\lambda_j^A\right|^2-a_j\left|\lambda_i^A\right|^2\right)}}{\left|\lambda_j^A\right|^2-\left|\lambda_i^A\right|^2}
  \text{ for some $1\leqslant i,j\leqslant N$ with $\left|\lambda_i^A\right|\neq\left|\lambda_j^A\right|$,}
\end{align}
then $\mu_1(A)=\frac{\Re\left\langle w,Aw\right\rangle}{|Aw|}$ for an antieigenvector $w\in\K^N$ with
\begin{align*}
  \left|w_i\right|^2 = \frac{a_j\left|\lambda_j^A\right|^2-2a_i\left|\lambda_j^A\right|^2+a_j\left|\lambda_i^A\right|^2}{\left(\left|\lambda_i^A\right|^2-\left|\lambda_j^A\right|^2\right)
        \left(a_i-a_j\right)},\quad
  \left|w_j\right|^2 = \frac{a_i\left|\lambda_i^A\right|^2-2a_j\left|\lambda_i^A\right|^2+a_i\left|\lambda_j^A\right|^2}{\left(\left|\lambda_i^A\right|^2-\left|\lambda_j^A\right|^2\right)
        \left(a_i-a_j\right)}
\end{align*}
and $\left|w_k\right|=0$ for $1\leqslant k\leqslant N$ with $k\neq i$ and $k\neq j$. This result can be found in \cite[Theorem 5.1]{GustafsonSeddighin1989}, 
\cite[Theorem 3.1]{GustafsonSeddighin1993}, \cite[Theorem 1.1]{GustafsonSeddighin2005} and \cite[Theorem 1]{Seddighin2005}. The proof in \cite[Theorem 5.1]{GustafsonSeddighin1989} 
is based on an application of the Lagrange multiplier method in order to solve a minimization problem.
%  - explicit representations: 
%      - GustafsonSeddighin1989 (Theorem 5.1, finite dimensional accretive normal operators), 
%      - GustafsonSeddighin1993 (Theorem 3.1, normal operator on a finite dimensional Hilbert space) 
%      - Seddighin2005 (Theorem 1, compact normal operator on seperable complex Hilbert space), 
%      - GustafsonSeddighin2005 (Theorem 1.1, accretive normal matrices, Theorem 1.2, finite dimensional strictly accretive normal operators)
%      - Davis1980 (Theorem, finite dimensional strictly accretive normal operators)
%      - Mirman1983 (Corollary 2, Hermitian positive definite operator, Corollary 1, strictly accretive normal operator)
%      - HornJohnson2013 (7.4.P4, Hermitian, positive definite) 
Furthermore, in \cite{Davis1980} it was proved that the expression on the right hand side of \eqref{equ:AntieigenvalueNormalMatrixPossibility2} is an upper bound 
for $\mu_1(A)$. In \cite{Davis1980} one can also find a geometric interpretation of this equality by a semi-ellipse.
  
If $\mu_1(A)$ is given by \eqref{equ:AntieigenvalueNormalMatrixPossibility1} for some $1\leqslant j\leqslant N$, then the $L^p$-antieigenvalue condition 
\eqref{cond:A4} is equivalent to, compare \eqref{equ:DCProp8},
\begin{align}
  \label{equ:AntieigenvalueNormalMatrixPossibility3}
  \frac{|p-2|}{2\sqrt{p-1}}\left|\Im\lambda_j^A\right| < \Re\lambda_j^A,\;1<p<\infty.
\end{align}
This leads to a \begriff{cone condition} which postulates that every eigenvalues of $A$ is even contained in a $p$-dependent sector 
$\Sigma_p$ in the open right half-plane, called a \begriff{conic section},
\begin{align*}
  \Sigma_p :=&\left\{\lambda\in\C\mid\frac{\left|p-2\right|}{2\sqrt{p-1}}\left|\Im\lambda\right|< \Re\lambda\right\} \\
            =&\left\{\lambda\in\C\mid\left|\arg\lambda\right|<\cos^{-1}\left(\frac{\left|p-2\right|}{p}\right)\right\},\,1<p<\infty,
\end{align*}
see Figure \ref{fig:LpDissipativityHermitianMatrices}(b). The opening angle $\left|\arg\lambda\right|$ is close to $0$ for small and large $p$, i.e. $p$ close to $1$ or $\infty$, 
and it is $\frac{\pi}{2}$ for $p=2$. Indeed, this is the same requirement as in the scalar complex case, compare \eqref{equ:DCProp9} for $N=1$ and $b=p-2$. In particular, to satisfy 
the cone condition for arbitrary large $p$, the matrix $A$ must be of the form $A=\mathrm{diag}(a_1,\ldots,a_N)$ for some positive $a_1,\ldots,a_N\in\R$.

If $\mu_1(A)$ is given by \eqref{equ:AntieigenvalueNormalMatrixPossibility2} for some $1\leqslant i,j\leqslant N$ with $\left|\lambda_i^A\right|\neq\left|\lambda_j^A\right|$, 
then the $L^p$-antieigenvalue condition \eqref{cond:A4} is equivalent to
\begin{align}
  \label{equ:AntieigenvalueNormalMatrixPossibility4}
  \frac{2\sqrt{\left(a_j-a_i\right)\left(a_i\left|\lambda_j^A\right|^2-a_j\left|\lambda_i^A\right|^2\right)}}{\left|\lambda_j^A\right|^2-\left|\lambda_i^A\right|^2}>\frac{|p-2|}{p},
  \;1<p<\infty.
\end{align}
We emphasize the following equalities from \cite[Section 6]{GustafsonSeddighin1989} and \cite{Davis1980}
\begin{align*}
   &\frac{2\sqrt{\left(a_j-a_i\right)\left(a_i\left|\lambda_j^A\right|^2-a_j\left|\lambda_i^A\right|^2\right)}}{\left|\lambda_j^A\right|^2-\left|\lambda_i^A\right|^2} \\
  =&\frac{2\sqrt{\frac{|\lambda_j^A|}{|\lambda_i^A|}\left[\left(\frac{a_i}{|\lambda_i^A|}\right)\left(\frac{|\lambda_j^A|}{|\lambda_i^A|}\right)-\frac{a_j}{|\lambda_j^A|}\right]\left[\left(\frac{a_j}{|\lambda_j^A|}\right)\left(\frac{|\lambda_j^A|}{|\lambda_i^A|}\right)-\frac{a_i}{|\lambda_i^A|}\right]}}{\left(\frac{\left|\lambda_j^A\right|}{\left|\lambda_i^A\right|}\right)^2-1} \\
  =&\frac{2\sqrt{\left(r_i\rho_{ij}-r_j\right)\left(r_j\rho_{ij}-r_i\right)\rho_{ij}}}{\rho_{ij}^2-1},
\end{align*}
where $\rho_{ij}:=\frac{|\lambda_j^A|}{|\lambda_i^A|}$ and $r_k:=\Re\frac{\lambda_k^A}{|\lambda_k^A|}=\frac{a_k}{|\lambda_k^A|}$ for $k=i,j$. This relation is helpful 
to verify, that all pairs of eigenvalues $\lambda_j^A$ and $\lambda_i^A$ satisfying \eqref{equ:AntieigenvalueNormalMatrixPossibility4} (under the constraint from the definition of $F$) 
must belong to a semi-ellipse, \cite{Davis1980}. Moreover, note that in the scalar complex case with $A=\alpha\in\C$ we have $E=\left\{\frac{\Re\alpha}{|\alpha|}\right\}$, 
$F=\emptyset$, which implies $\frac{\Re\alpha}{|\alpha|}=\mu_1(\alpha)>\frac{|p-2|}{p}$. This is equivalent to \eqref{equ:DCProp8} and also to \eqref{equ:DCProp9}.

% 4.5. \mu_1(A) for arbitrary invertible matrices
\subsection{\texorpdfstring{$\boldsymbol{\mu_1(A)}$}{mu1(A)} for arbitrary matrices}\label{sec:4.5}
If $A$ is an arbitrary matrix, there are only approximation results for $\mu_1(A)$. Such results are rather new in the literature and can be found in \cite[Theorem 2]{Seddighin2012}.
However, for an arbitrary given matrix $A$ it is also possible to compute the first antieigenvalue and its corresponding antieigenvector directly. The computation of antieigenvalues 
and antieigenvectors has been analyzed in \cite{Seddighin2005a,Seddighin2003}.

%\bibliographystyle{abbrv}
%\bibliography{literature}

\def\cprime{$'$}

\end{document}